\documentclass [reqno] 
{amsart}

\copyrightinfo{2000}{American Mathematical Society}

\newtheorem{theorem}{Theorem}[section]
\newtheorem{lemma}[theorem]{Lemma}
\newtheorem{corollary}[theorem]{Corollary}

\theoremstyle{definition}
\newtheorem{definition}[theorem]{Definition}

\theoremstyle{remark}
\newtheorem{remark}[theorem]{Remark}

\numberwithin{equation}{section}

\newcommand{\C}{\mathbb{C}}
\newcommand{\D}{\mathbb{D}}

\begin{document}

\title[Separation of Boundary Singularities]
      {Separation of Boundary Singularities for Holomorphic Generators}

\author[M. Elin]{Mark Elin}

\address{Department of Applied Mathematics,
         ORT Braude College,
         P.0. Box 78,
         20101 Karmiel,
         Israel}

\email{mark\_elin@braude.ac.il}

\thanks{This research is part of the European Science Foundation Networking Programme HCAA}

\author[D. Shoikhet]{David Shoikhet}

\address{Department of Applied Mathematics,
         ORT Braude College,
         P.0. Box 78,
         20101 Karmiel,
         Israel}

\email{davs@braude.ac.il}

\thanks{The second author gratefully acknowledges the support of
        the Deutsche Forschungs\-gemeinschaft.}

\author[N. Tarkhanov]{Nikolai Tarkhanov}

\address{Institute of Mathematics,
         University of Potsdam,
         Am Neuen Palais 10,
         14469 Potsdam,
         Germany}

\email{tarkhanov@math.uni-potsdam.de}

\date{December 23, 2009}




\begin{abstract}
We prove a theorem on separation of boundary null points for generators of
continuous semigroups of holomorphic self-mappings of the unit disk in the
complex plane.
Our construction demonstrates the existence and importance of a particular
role of the binary operation $\circ$ given by $1 / f \circ g = 1/f + 1/g$
on generators.
\end{abstract}

\maketitle

\tableofcontents

\section{Introduction}
\label{s.Introduction}

A flow is a one-parameter semigroup of transformations acting on a set
$\mathcal{X}$ called the phase space of the flow.
Typically, the phase space is endowed with an additional structure which the
transformations are required to respect.
In other words, associated to each $t \geq 0$ there is a self-mapping $g_t$ of
$\mathcal{X}$, such that
   $g_0$ is the identity mapping of $\mathcal{X}$
and
   $g_{s+t} = g_s g_t$ for all $s, t \geq 0$.
A cascade differs from a flow in that the mappings $g_t$ are only defined for
$t = 0, 1, \ldots$.
In this case, one uses $n$ instead of $t$.

The additional structure in the phase space is
   either the structure of a topological space
   or of a manifold.
In the first case, it is required that $g_t x$ be continuous with respect to
$(t,x)$. In the second case, smooth dynamical systems are defined.
A cascade is smooth if all the $g_n$ are smooth mappings.
A flow is smooth if $g_t x$ depends smoothly on $(t,x)$ and
   $v (x) = (d/dt) g_t x\, |_{t=0}$
is a smooth vector field on $\mathcal{X}$.
The latter completely determines the flow.
For a fixed $x_0$ and variable $t$, $g_t x_0$ is a solution of
   the differential equation $\dot{x} = v (x)$ with
   initial condition $x (0) = x_0$;
that is, $g_t x_0 = x (t)$,
   see \cite{Anos88}.

A complex dynamical system is defined as any one-parameter semigroup of
holomorphic self-mappings of a complex space $\mathcal{X}$.
Complex dynamical systems are not only of intrinsic interest in mathematics,
they are also of great significance in diverse applications, e.g.,
   in the theory of branching processes \cite{Harr63}, \cite{Seva71},
   in the geometry of Banach spaces \cite{Araz87},
   in control theory and optimization \cite{HelmMoor94} and
   in the theory of composition operators \cite{Shap93}.

If $\mathcal{X}$ is of finite dimension then each continuous complex dynamical
system on $\mathcal{X}$ must be differentiable,
   see \cite{Abat88}.
In this paper, Abate also proves a condition on a holomorphic mapping $f$ of
$\mathcal{X}$ which is necessary and sufficient for $f$ to be a generator of
some complex flow on $\mathcal{X}$.
Similar results had already been published in \cite{BerkPort78} in the
one-dimensional case.

We recall that holomorphic self-mappings of a domain $\mathcal{X}$ in a
complex Banach space are nonexpanding with respect to each pseudometric
generated by a Schwarz-Pick system \cite{Harr79}.
Thus, the question arises whether a fixed point theory like that for monotone
and nonexpanding operators can be developed for holomorphic mappings.
In the one-dimensional case the classical Denjoy-Wolff theorem yields not only
information on the configuration of fixed points but also on the behaviour of
iterates of a holomorphic self-mapping.
In \cite{Kubo83},
   \cite{MacC83},
   \cite{Merc93}
the Denjoy-Wolff theory is developed for higher-dimensional complex spaces
under certain assumptions on domains and self-mappings.

In this paper we study complex dynamical systems on the unit disk $\D$ in the
complex plane $\C$.
The generators of continuous semigroups of holomorphic self-mappings of $\D$
constitute a very special class of holomorphic functions in the disk,
   see \cite{BerkPort78}.
It has the structure of a real cone. The regular null points of
generators are common fixed points for all mappings in the
corresponding semigroups. Among them are the so-called
Denjoy-Wolff points of semigroups which may also lie on the
boundary of the disk. Generators may have boundary regular null
points other than the Denjoy-Wolff points.

We focus on boundary regular null points of holomorphic generators $f$, which
are singular points of the vector fields $f$.
Let
   $f$ be the generator of a complex flow on $\D$ with the Denjoy-Wolff point
   $a = 1$
(and so $f (a) = 0$ and $f' (a) \geq 0$).
Suppose that the angular limit values of $f$ also vanish at some points
   $\zeta_1, \ldots, \zeta_N$
on $\partial \D$.
Associated to each pair $(a,\zeta_n)$ there is a unique generator $g_n$ of a
complex one-parameter group on $\D$ with fixed points $a$ and $\zeta_n$, such
that the function $h$ defined by
   $1/h = 1/f - 1/g_n$
is a generator of a complex semigroup on the disk which has no singularity at
$\zeta_n$, for $n = 1, \ldots, N$.
On repeating this argument $N$ times we represent $1/f$ as the sum of all
   $1/g_n$
modulo $1/h$, where
   $h$ is a holomorphic generator on $\D$ of the same type as $f$ but without
   singularities at $\zeta_1, \ldots, \zeta_N$.

Note that these considerations are also closely related to the boundary
Nevan\-linna-Pick interpolation problem for holomorphic mappings
   (see, for example, \cite{Sa} and \cite{BH}).

We next establish distortion theorems which give very sharp estimates for the
remainder
   $1/f - \sum 1/g_n$
in terms of $f' (\zeta_n)$.

In parallel, we prove similar results for generators of flows on $\D$ with
interior Denjoy-Wolff points.

\section{Semigroups of analytic functions}
\label{s.soaf}

Let $\D$ be the open unit disk around the origin in the complex plane $\C$.
Given an open set $\mathcal{W}$ in $\C$, we denote by
   $H (\D,\mathcal{W})$
the set of holomorphic functions in $\D$ with values in $\mathcal{W}$.
If $\mathcal{W} = \C$, we simply write
   $H (\D)$
for the set of all holomorphic functions in $\D$.

A family
   $\mathcal{S} = \{ F_n \}_{n = 0, 1, \ldots}$
of holomorphic self-mappings of $\D$ is said to be a one-parameter semigroup
with discrete time if
   1) $F_{m+n} = F_m \circ F_n$ for all $m, n = 0, 1, \ldots$,
and
   2) $F_0 (z) = z$.
Such a semigroup actually consists of the iterates of $F = F_1$, for
   $F_0 = I$, the identity mapping of $\D$, because of 2),
and
   $F_n = F \circ F_{n-1}$ for all $n = 1, 2, \ldots$, because of 1).

\begin{definition}
\label{d.semigroup}
A family
   $\mathcal{S} = \{ F_t \}_{t \geq 0}$
of holomorphic self-mappings of $\D$ is called a one-parameter semigroup with
continuous time if
   1) $F_{s+t} = F_s \circ F_t$ for all $s, t \geq 0$, and
   2) $F_0 (z) = z$.
\end{definition}

For a one-parameter semigroup with continuous time, the mere continuity on the
right at $t = 0$ implies the differentiability, hence, the continuity, in $t$
on all of $[0,\infty)$
   (see \cite{BerkPort78}).

The (infinitesimal) generator of a semigroup
   $\mathcal{S} = \{ F_t \}_{t \geq 0}$,
is understood to be a function $f \in H (\D)$ defined by
\begin{equation}
\label{eq.generator}
   f (z)
 := \lim_{t \to 0^+} \frac{1}{t} \left( z - F_t (z) \right).
\end{equation}

The semigroup can be restored from its generator $f$ in the following way.
For any $z \in \D$, find a solution $z (t)$ of the initial problem
   $\dot{z} (t) + f (z (t)) = 0$, if $t \geq 0$, and
   $z (0) = z$.
Then $F_t (z): = z (t)$ for $t \geq 0$. Note that if both $f$ and
                  $-f$
are generators on the disk $\D$, then the semigroup
   $\mathcal{S} = \{ F_t \}_{t \geq 0}$
generated by $f$ can be extended to a one-parameter group of automorphisms of
$\D$ with the property $(F_t)^{-1} = F_{-t}$ for all $t \in \mathbb{R}$.

The following characterisation of semigroup generators can be found in
   \cite{BerkPort78} and
   \cite{AharReicShoi99}
(see also \cite{Shoi01}).

\begin{theorem}
\label{t.generator}
For every holomorphic function $f\,$ in $\D\,$, the following are equivalent:

1)
$f$ is a semigroup generator on $\D$.

2)
$f$ admits the representation
   $f (z) = (z - a) (1 - \bar{a} z) p (z)$,
where
   $a \in \overline{\D}$
and
   $p$ is a holomorphic function in $\D$ with nonnegative real part.

3)
$f$ is of the form
   $f (z) = c - \bar{c} z^2 + z q (z)$,
where
   $q$ is a holomorphic function in $\D$ with nonnegative real part.
Moreover, $f$ generates a group of automorphisms of $\D$ if and only if
   $\Re q (z) \equiv 0$.
\end{theorem}

The equivalence of 1) and
                   2)
is due to
   \cite{BerkPort78},
and the part 3) was proved in \cite{AharReicShoi99}.

If a semigroup $\mathcal{S}$ (with either discrete or continuous time)
   is not trivial (i.e. it does not reduce to the only identity mapping)
and
   contains no elliptic automorphisms of $\D$,
then there is a unique point $a$ in the closure of $\D$, such that
   $F_t (z) \to a$ as $t \to \infty$,
for all $z \in \D$.
This point $a$ is called the Denjoy-Wolff point of $\mathcal{S}$.
For continuous semigroups, the point
   $a \in \overline{\D}$
in the representation 2) of Theorem \ref{t.generator} is exactly the
Denjoy-Wolff point of $\mathcal{S}$.
Hence it follows that this representation is unique.

In fact, it is sufficient to consider two distinct cases, namely
   $a = 0$ (the interior or dilation case) and
   $a = 1$ (the boundary case).

For each $F \in H (\D,\D)$, the mapping $f = I - F$ is a generator of a
continuous semigroup
   $\mathcal{S} = \{ F_t \}_{t \geq 0}$,
see for instance \cite{ReicShoi05}.
Note that $F_1 \neq F$, however, the set of fixed points of $F_1$
                                               (and that of $F$, as is clear)
coincides with the set of null points of $f$.

If
   $h$ is a holomorphic function in $\D$ and
   $\zeta \in \partial \D$,
then
$$
   \angle \lim_{z \to \zeta} h (z) =: h (\zeta)
$$
for the angular (or nontangential) limit of $h$ at $\zeta$,
   see e.g., \cite{Pomm92}.
In \cite{ElinShoi01} it was shown that if $a \in \partial \D$ is the
Denjoy-Wolff point of a continuous semigroup $\mathcal{S}$ generated by $f$
then the angular derivative
\begin{eqnarray*}
   f' (a)
 & := &
   \angle \lim_{z \to a} f' (z)
\\
 & = &
   \angle \lim_{z \to a} \frac{f (z)}{z-a}
\end{eqnarray*}
exists and is a nonnegative real number, i.e. $f' (a) \geq 0$.

If $f' (a) = 0$, then $f$ (and $\mathcal{S}$) is said to be of parabolic type.
Otherwise, if $f' (a) > 0$, then $f$ (respectively, $\mathcal{S}$) is said to
be of hyperbolic type.

In general, a point $\zeta \in \partial \D$ is called a boundary regular null
point of $f \in H (\D)$ if
   the angular limit of $f$ at $\zeta$ (exists and) is equal to $0$
and
   the angular derivative of $f$ at $\zeta$ (exists and) is finite.
Similarly, a point $\zeta \in \partial \D$ is called a boundary regular fixed
point of $F \in H (\D,\D)$ if
   the angular limit of $F$ at $\zeta$ (exists and) is equal to $\zeta$
and
   the angular derivative of $F$ at $\zeta$ (exists and) is finite.
Thus, the boundary Denjoy-Wolff point of a semigroup is a boundary regular
null point of the semigroup generator.
More generally, it follows from
   \cite{ElinShoi01},
   \cite{Shoi03} and
   \cite{ContDiaz05}
that each boundary regular null point of a generator $f$ is a boundary regular
fixed point of each mapping $F_t$ of the generated semigroup, and
   $F_t' (\zeta) = \exp (- f' (\zeta) t)$,
with $f' (\zeta)$ being a real number.

\section{Remarks on Cowen-Pommerenke type inequalities}
\label{s.roCPti}

Recall again that if
   $a \in \overline{\D}$ is the Denjoy-Wolff point of a single holomorphic
   self-mapping $F$ of $\D$
then
   either $a \in \D$ and $|F' (a)| < 1$
   or $a \in \partial \D$ and $0 < F' (a) \leq 1$,
which is due to the classical Schwarz and Julia's lemmas.
So, for each boundary regular fixed point $\zeta$ of $F$ different from $a$,
we get $F' (\zeta) > 1$.
The following result is proved in \cite{CowePomm82}.

\begin{theorem}
\label{t.CowePomm82}
Let
   $F$ be a holomorphic self-mapping of the unit disk $\D$ with the
   Denjoy-Wolff point $a \in \overline{\D}$
and let
   $\zeta_1, \ldots, \zeta_N$ be boundary regular fixed points of $F$
   different from $a$.

1)
If $a = 0$ (dilation case), then
$$
   \sum_{n=1}^N \frac{1}{F' (\zeta_n) - 1}
 \leq \Re \Big( \frac{1 + F' (0)}{1 - F' (0)} \Big).
$$

2)
If $a = 1$ and $F' (a) < 1$ (hyperbolic case), then
$$
   \sum_{n=1}^N \frac{1}{F' (\zeta_n) - 1}
 \leq \frac{F' (1)}{1 - F' (1)}.
$$

3)
If $a = 1$ and $F' (a) = 1$ (parabolic case), then
$$
   \sum_{n=1}^N \frac{|1 - \zeta_n|^2}{F' (\zeta_n) - 1}
 \leq 2\, \Re \Big( \frac{1}{F (0)} - 1 \Big).
$$
Moreover, equality holds if and only if $F$ is a Blaschke product
   of order $N+1$ in the case 1)
or
   of order $N$ in the cases 2) and 3).
\end{theorem}

Using Theorem \ref{t.CowePomm82} one establishes an important sharp inequality
for contact points of a holomorphic self-mapping $F$ of $\D$,
   cf. \textit{ibid}.
Namely, assume that
   there are boundary points $\zeta_1, \ldots, \zeta_N$,
   such that $F (\zeta_n) = w$ for all $n = 1, \ldots, N$,
   with $w \in \partial \D$.
Then
\begin{equation}
\label{eq.CowePomm82}
   \sum_{n=1}^N \frac{1}{|F' (\zeta_n)|}
 \leq \Re \Big( \frac{w + F (0)}{w - F (0)} \Big).
\end{equation}
Moreover, equality is valid if and only if $F$ is a Blaschke product of order
$N$.

In \cite{ContDiazPomm06}, a geometrical approach and
                            generating theory
for continuous semigroups is developed to derive another sharp inequality for
functions with contact points.
We present it as follows.

\begin{theorem}
\label{t.sharp}
Let
   $F$ be a holomorphic self-mapping of $\D$,
such that
   $F (1) = 1$ and
   $F (\zeta_n) = w$ for some boundary points $\zeta_1, \ldots, \zeta_N$,
with $w \neq 1$.
Suppose that the angular derivatives $F' (1)$ and
                                     $F' (\zeta_1), \ldots, F' (\zeta_N)$
are finite.
Then
$$
   \sum_{n=1}^N \frac{1}{|F' (\zeta_n)|} \frac{1 - \Re w}{1 - \Re \zeta_n}
 \leq F' (1),
$$
and equality holds if and only if
$
   \displaystyle
   F (z) = \frac{1 + (1 - w) A (z)}{1 - (1 - w) A (z)},
$
where
$$
   A (z)
 = \sum_{n=1}^N
   \frac{1}{|F' (\zeta_n)|} \frac{\zeta_n}{1-\zeta_n} \frac{z-1}{z-\zeta_n}.
$$
\end{theorem}

An open question still remaining is if the extremal functions in
   (\ref{eq.CowePomm82}) and
   Theorem \ref{t.sharp}
are the same.
This question is handled in Section \ref{s.bDWp}.

For continuous semigroups of holomorphic self-mappings of the unit disk, an
infinitesimal version of Theorem \ref{t.CowePomm82} related to a boundary
Denjoy-Wolff point is given in
   \cite{ContDiazPomm06}
in a unified form which includes both hyperbolic and parabolic cases.

\begin{theorem}
\label{t.ContDiazPomm06}
Let
   $\mathcal{S} = \{ F_t \}_{t \geq 0}$
be a semigroup in $H (\D,\D)$
   with a generator $f \in H (\D)$ and the Denjoy-Wolff point $a = 1$,
and let
   $\zeta_1, \ldots, \zeta_N$ be boundary regular null points of $f$
   different from $a$,
i.e.
   $f (\zeta_n) = 0$ and
   $f'(\zeta_n) < 0$.
Then
\begin{equation}
\label{eq.ContDiazPomm06}
   \sum_{n=1}^N \frac{1 - \Re \zeta_n}{|f' (\zeta_n)|}
 \leq - \Re \Big( \frac{1}{f (0)} \Big).
\end{equation}
Moreover, equality holds if and only if
   $F_t (z) = \sigma^{-1} (\sigma (z) + t)$ for $t \geq 0$,
where
$$
   \sigma (z)
 = \frac{1}{f (0)} \frac{z}{z-1}
 + 2
   \sum_{n=1}^N
   \frac{1 - \Re \zeta_n}{|f' (\zeta_n)|}
   \Big( \frac{\zeta_n}{(1 - \zeta_n)^2} \log \frac{1-z}{1 - \bar{\zeta}_n z}
       + \frac{\bar{\zeta}_n}{1 - \bar{\zeta}_n} \frac{z}{1 - z}
   \Big).
$$
\end{theorem}

The proof of this theorem is based on the profound study of a linearisation
model for semigroups of holomorphic self-mappings with boundary Denjoy-Wolff
points, which is given by Abel's functional equation
\begin{equation}
\label{eq.Abel}
   \sigma (F_t (z)) = \sigma (z) + t,
\end{equation}
and nice geometric properties of its solution.

The principal significance of (\ref{eq.Abel}) is that it allows the study of
general properties of univalent (i.e. one-to-one)
   functions convex in one direction, as well as
   functions starlike with respect to a boundary point,
and extremal problems for these function classes.
The starlike functions can be obtained as solutions to the so-called
Schr\"{o}der functional equation related to semigroups of hyperbolic type,
cf. for instance \cite{ElinShoi06} and
                 \cite{ElinShoiZalc08}.

Theorem \ref{t.ContDiazPomm06} does not cover the dilation case in which the
Denjoy-Wolff point is inside $\D$.
In addition, in the hyperbolic case the unified form does not allow the use of
some very nice tools based on the fact that $f' (a) > 0$ which is not taken
into consideration in inequality (\ref{eq.ContDiazPomm06}).
On the other hand,
   for parabolic type semigroups,
inequality (\ref{eq.ContDiazPomm06}) can be obtained directly from the part 3)
of Theorem \ref{t.CowePomm82} by using the definition of generators and
                                       Theorem 1 of \cite{ContDiazPomm06}
(see also \cite{ElinShoi06}).

More precisely, let
   $\mathcal{S} = \{ F_t \}_{t \geq 0}$
be a semigroup of parabolic type generated by a function $f \in H (\D)$ with
   $f (1) = 0$ and
   $f' (1) = 0$.
Assuming
   $f (\zeta_n) = 0$ and
   $f' (\zeta_n) < 0$
for some points $\zeta_1, \ldots, \zeta_n$ on $\partial \D$ different from $1$,
we get by Theorem \ref{t.CowePomm82} (part 3))
$$
   \sum_{n=1}^N \frac{|1 - \zeta_n|^2}{F_t' (\zeta_n) - 1}
 \leq 2\, \Re \Big( \frac{1}{F_t (0)} - 1 \Big),
$$
or
$$
   \sum_{n=1}^N \frac{t\, |1 - \zeta_n|^2}{\exp (- t f' (\zeta_n)) - 1}
 \leq 2\, \Re \Big( \frac{t}{F_t (0)} - t \Big)
$$
for all $t \geq 0$.
Letting $t \to 0$ and
using (\ref{eq.generator}) we get
$$
   \sum_{n=1}^N \frac{|1 - \zeta_n|^2}{- f' (\zeta_n)}
 \leq - 2\, \Re \Big( \frac{1}{f (0)} \Big),
$$
which coincides with (\ref{eq.ContDiazPomm06}).
The following infinitesimal analogue of Theorem \ref{t.CowePomm82} can be
proven similarly.

\begin{corollary}
\label{c.CowePomm82}
Let
   $\mathcal{S} = \{ F_t \}_{t \geq 0}$
be a semigroup of holomorphic self-mappings of the disk $\D$ generated by $f$
with
   $f (a) = 0$
and
   $\Re f' (a) \geq 0$ for some $a \in \overline{\D}$,
and let
   $\zeta_1, \ldots, \zeta_N$ be boundary regular null points of $f$ different
   from $a$.
The following is valid:

1)
If $a = 0$, then
$$
   \sum_{n=1}^N \frac{1}{|f' (\zeta_n)|}
 \leq 2\, \Re \Big( \frac{1}{f' (0)} \Big).
$$

2)
If $a = 1$ and $0 < f' (a) < \infty$, then
$$
   \sum_{n=1}^N \frac{1}{|f' (\zeta_n)|}
 \leq \frac{1}{f' (1)}.
$$

3)
If $a = 1$ and $f' (a) = 0$, then
$$
   \sum_{n=1}^N \frac{1 - \Re \zeta_n}{|f' (\zeta_n)|}
 \leq - \Re \Big( \frac{1}{f (0)} \Big).
$$
\end{corollary}

This method appears difficult to derive the extremal forms of
generators under which equalities in the assertions 1)-3) hold. In
this paper we develop {\it inter alia} other analytic and
geometric approaches to establish the desirable results.

If
   $\mathcal{S} = \{ F_t \}_{t \geq 0}$
is a semigroup in $H (\D,\D)$ with a generator $f \in H (\D)$ and
                                   a Denjoy-Wolff point $a \in \overline{\D}$,
then $f (a) = 0$ and
     $\Re f' (a) \geq 0$.
Hence Corollary \ref{c.CowePomm82} applies to yield a modification of
   Theorem \ref{t.ContDiazPomm06}
which also contains the dilation case not considered previously.
More precisely, assume that
   $\zeta_1, \ldots, \zeta_N$ are boundary regular null points of $f$
   different from $a$.

a)
If $a = 0$, then
$$
   \sum_{n=1}^N \frac{1}{|f' (\zeta_n)|}
 \leq 2\, \Re \Big( \frac{1}{f' (0)} \Big).
$$

b)
If $a = 1$, then
$$
   \sum_{n=1}^N \frac{1 - \Re \zeta_n}{|f' (\zeta_n)|}
 \leq - \Re \Big( \frac{1}{f (0)} \Big).
$$
Moreover, equality in a) or
                      b)
holds if and only if $f$ is of the form
\begin{equation}
\label{eq.extremal}
\begin{array}{rcl}
   f (z) & = & \displaystyle z\, \frac{1 - F (z)}{1 + F (z)},
\\
   f (z) & = & \displaystyle - (1-z)^2\, \frac{1 - F (z)}{1 + F (z)},
\end{array}
\end{equation}
respectively, $F$ being a Blaschke product of order $N$.
Indeed, it follows from the Berkson-Porta formula that $f$ can be represented
by
   the first formula of (\ref{eq.extremal}) in the case a)
and by
   the second formula of (\ref{eq.extremal}) in the case b),
where
   $F$ is a holomorphic self-mapping of $\D$.
Since $f' (\zeta_n)$ exists and is finite, we immediately conclude
that
   $F (\zeta_n) = 1$
for all $n = 1, \ldots, N$.
Then our assertion follows from (\ref{eq.CowePomm82}) by simple computation.

Observe that since each element of a continuous semigroup
   $S = \left\{ F_t \right\}_{t \geq 0}$
is a univalent function on $\mathbb{D}$, one can invoke certain results of
Cowen-Pommerenke type, to get appropriate estimates for generators.
For example, in the dilation case (i.e. $a = 0$) a recent result for univalent
functions given in Corollary 4.1 of \cite{AV08} leads to the inequality
$$
   \sum_{n=1}^N \frac1{|f'(\zeta_n)|}
 \leq
   \frac{2}{\Re f'(0)}
$$
for holomorphic generators, which is obviously weaker than the one
given in the assertion a) above. The point is that univalent
self-mappings of $\mathbb{D}$ may not be embedded in general into
a continuous semigroup.

Note also that the assertion b) does not distinguish between hyperbolic and
parabolic cases.
Indeed, (\ref{eq.CowePomm82}) applies
   neither to inequality 2) of Corollary \ref{c.CowePomm82}
   nor to its extremal function which fulfills the equality.
Therefore, in the hyperbolic case the natural question arises of whether the
extremal function
   which satisfies equality in 2) of Corollary \ref{c.CowePomm82}
is actually the same as the second function of (\ref{eq.extremal})
   fulfilling equality in the assertion b),
provided $f' (1) > 0$.

In the next section we present a decomposition theorem for generators with
boundary Denjoy-Wolff points.
It enables us to recover the inequality 2) of Corollary~\ref{c.CowePomm82} as
well as some quantitative algebraic and geometric characteristics related to
this case.
This theorem may be thought of as separation of singularities for such
generators.
A general result on separation of singularities is given in
Section~\ref{s.separation}, which also includes the dilation case.
Yet another look at the problem in question leads to an infinitesimal version
of boundary interpolation theorem for holomorphic generators a la Pick and
Nevanlinna.
Roughly speaking, this problem consists in the following.
Given
   points $\zeta_1, \ldots, \zeta_{N+1}$ on $\partial \mathbb{D}$
and
   real numbers $m_1, \ldots, m_{N+1}$,
find all holomorphic generators satisfying
   $f (\zeta_n) = 0$ and
   $f' (\zeta_n) = m_n$ for $n = 1, \ldots, N+1$.
In Section~\ref{s.distortion} we also give distortion theorems for generators
of all types which automatically include
   inequalities 1)--3) of Corollary \ref{c.CowePomm82} and
   extremal functions to be given in Sections \ref{s.bDWp} and
                                              \ref{s.separation}.

\section{Boundary Denjoy-Wolff points}
\label{s.bDWp}

Let $g$ be a generator of a group of hyperbolic automorphisms of $\D$ having
   attracting (Denjoy-Wolff) fixed point $a = 1$
and
   a repelling fixed point $\zeta \in \partial \D$.
From the part 3) of Theorem \ref{t.generator} it follows that
$$
   g (z)
 = c\, \frac{(z-1) (z-\zeta)}{1-\zeta}
$$
for some $c > 0$, and so
$
   \displaystyle
   \frac{1}{g (z)}
 = \frac{1}{c} \Big( \frac{1}{z-1} - \frac{1}{z-\zeta} \Big).
$

The following theorem shows that for any generator $f$ with boundary regular
null points
   $\zeta_1, \ldots, \zeta_N$, either
$$
   \frac{1}{f (z)}
 = \sum_{n=1}^{N} \frac{1}{g_n (z)},
$$
where
   $g_n$ is a group generator with the Denjoy-Wolff point $a = 1$ and the
   repelling point $\zeta = \zeta_n$,
or
   there is a generator $h$ of the same (hyperbolic or parabolic) type as $f$,
   such that $\zeta_1, \ldots, \zeta_N$ are not regular null points of $h$ and
\begin{equation}
\label{eq.preliminary}
   \frac{1}{f (z)} - \frac{1}{h (z)}
 = \sum_{n=1}^{N} \frac{1}{g_n (z)}.
\end{equation}

\begin{theorem}
\label{t.bDWp}
Let
   $f \in H (\D)$ be the generator of a semigroup with the Denjoy-Wolff point
   $a = 1$.
Assume that $f$ has boundary regular null points $\zeta_1, \ldots, \zeta_N$
different from $1$
   (each $f' (\zeta_n)$ being negative).
Then there exists a number $r \geq 0$, such that
\begin{equation}
\label{eq.bDWp}
   \frac{1}{f (z)}
 = \sum_{n=1}^{N}
   \frac{1}{|f' (\zeta_n)|}
   \Big( \frac{1}{z-1} - \frac{1}{z-\zeta_n} \Big)
 + \frac{r}{h (z)},
\end{equation}
where $h$ is a holomorphic generator on $\D$ satisfying $h' (1) = f' (1)$ and
$$
   \angle \lim_{z \to \zeta_n} \frac{h (z)}{z-\zeta_n} = \infty
$$
for all $n = 1, \ldots, N$.
\end{theorem}

\begin{proof}
This theorem follows by induction from its reduced form corresponding to
$N = 1$.
More precisely, let
   $f \in H (\D)$ be the generator of a semigroup with the Denjoy-Wolff point
   $a = 1$ and a boundary regular fixed point $\zeta \neq 1$
(so $f' (\zeta) < 0$).
Suppose that
$$
   \frac{1}{f (z)}
 \neq
   \frac{1}{|f' (\zeta)|}
   \Big( \frac{1}{z-1} - \frac{1}{z-\zeta} \Big).
$$
Then the function $h$ defined by
\begin{equation}
\label{eq.h(z)}
   \frac{1}{h (z)}
 = \frac{1}{f (z)}
 - \frac{1}{|f' (\zeta)|} \Big( \frac{1}{z-1} - \frac{1}{z-\zeta} \Big)
\end{equation}
generates a semigroup with the same Denjoy-Wolff point $a = 1$ and of the same
(hyperbolic or parabolic) type as $f$, and
$$
   \angle \lim_{z \to \zeta} \frac{h (z)}{z-\zeta} = \infty.
$$

To prove this assertion we need a modified version of the
   Julia-Wolff-Carath\'{e}odo\-ry theorem.

\begin{lemma}
\label{l.JWC}
Let
   $g$ be a holomorphic function in $\D$ with nonnegative 
real part.
Then, for each $\zeta \in \partial \D$, the limit
$$
   \angle \lim_{z \to \zeta} \frac{1}{2}\, (1 - \bar{\zeta} z) g (z)
 = \ell
$$
exists and is a nonnegative real number. Moreover, $
   \displaystyle
   \Re g (z) \geq \ell\, \frac{1-|z|^2}{|z-\zeta|^2}.
$
\end{lemma}

\begin{proof}
Consider the holomorphic function $G$ defined by
$$
   G (w) = g \Big( \zeta\, \frac{w-1}{w+1} \Big),
$$
or, equivalently,
$
   \displaystyle
   g (z) = G \Big( \frac{1 + \bar{\zeta} z}{1 - \bar{\zeta} z} \Big).
$
Then $G$ is a self-mapping of the right half-plane $\{ \Re w \geq 0 \}$ and
\begin{eqnarray*}
   \angle \lim_{w \to \infty} \frac{G (w)}{w+1}
 & = &
   \angle \lim_{w \to \infty}
   \frac{\displaystyle g \Big( \zeta\, \frac{w-1}{w+1} \Big)}
        {w+1}
\\
 & = &
   \angle \lim_{z \to \zeta}
   \frac{g (z)}
        {\displaystyle \frac{1 + \bar{\zeta} z}{1 - \bar{\zeta} z} + 1}
\\
 & = &
   \angle \lim_{z \to \zeta} \frac{1}{2}\, (1 - \bar{\zeta} z) g (z)
\\
 & = &
   \ell.
\end{eqnarray*}
By the Julia-Wolff-Carath\'{e}odory theorem,
$\ell$ is a nonnegative real number and $\Re G (w) \geq \ell\, \Re w$, or
$$
   \Re g (z) \geq \ell\, \frac{1-|z|^2}{|z-\zeta|^2},
$$
as desired.
\end{proof}

We proceed to prove Theorem \ref{t.bDWp}.
Consider the holomorphic function
$$
   g (z) = - \frac{(z-1)^2}{f (z)}
$$
in $\D$.
By the Berkson-Porta representation, $\Re g (z) \geq 0$ for all $z \in \D$,
   cf. Theorem \ref{t.generator}.
Furthermore,
$$
   \angle \lim_{z \to \zeta} (1 - \bar{\zeta} z) g (z)
 =
   2\, \frac{1 - \Re \zeta}{|f' (\zeta)|},
$$
that is, the function $g$ satisfies the assumptions of Lemma~\ref{l.JWC} with
$
   \displaystyle
   \ell = \frac{1 - \Re \zeta}{|f' (\zeta)|}.
$
Therefore,
\begin{equation}
\label{eq.rpg(z)}
   \Re g (z)
 \geq \frac{1 - \Re \zeta}{|f' (\zeta)|}\, \frac{1-|z|^2}{|z-\zeta|^2}.
\end{equation}

On the other hand, using (\ref{eq.h(z)}) yields
\begin{eqnarray*}
   \Re g (z)
 & = &
   - \Re \Big(
   \frac{(z-1)^2}{h (z)}
 + \frac{(z-1)^2}{|f' (\zeta)|} \Big( \frac{1}{z-1} - \frac{1}{z-\zeta} \Big)
         \Big)
\\
 & = &
   - \Re \Big( \frac{(z-1)^2}{h (z)} \Big)
   - \Re \Big( \frac{z-1}{|f' (\zeta)|}\, \frac{1-\zeta}{z-\zeta} \Big)
\\
 & = &
   - \Re \Big( \frac{(z-1)^2}{h (z)} \Big)
   + \frac{1 - \Re \zeta}{|f' (\zeta)|}\, \frac{1-|z|^2}{|z-\zeta|^2}.
\end{eqnarray*}
Comparing this expression with (\ref{eq.rpg(z)}) we conclude that
$$
 - \Re \Big( \frac{(z-1)^2}{h (z)} \Big) \geq 0,
$$
and so the Berkson-Porta formula shows that $h$ is a semigroup generator on
$\D$.
By (\ref{eq.h(z)}),
$$
   \angle \lim_{z \to 1} \frac{z-1}{h (z)}
 = \angle \lim_{z \to 1} \frac{z-1}{f (z)} - \frac{1}{|f' (\zeta)|}.
$$
Hence, the angular derivative of $h$ at $a = 1$ is different from zero if and
only if the angular derivative of $f$ is.

This proves the desired reduced form of Theorem \ref{t.bDWp}, implying that
in our situation either
$$
   \frac{1}{f (z)}
 = \sum_{n=1}^{N}
   \frac{1}{|f' (\zeta_n)|}
   \Big( \frac{1}{z-1} - \frac{1}{z-\zeta_n} \Big),
$$
i.e.,
   (\ref{eq.bDWp}) holds with $r = 0$ and arbitrary $h$,
or
$$
   \frac{1}{f (z)}
 = \sum_{n=1}^{N}
   \frac{1}{|f' (\zeta_n)|}
   \Big( \frac{1}{z-1} - \frac{1}{z-\zeta_n} \Big)
 + \frac{1}{H (z)}
$$
for some semigroup generator $H$.
It follows from Corollary \ref{c.CowePomm82} that
$$
   r = 1 - f' (1) \sum_{n=1}^{N} \frac{1}{|f' (\zeta_n)|}
     \geq 0.
$$
Since the set of all generators is a real cone, one can set
   $h (z) = r\, H (z)$
to get (\ref{eq.bDWp}) with $h' (1) = f' (1)$.
\end{proof}

As but one consequence we recover Theorem 2 of \cite{ContDiazPomm06} with the
extremal function given explicitly, cf.
   Theorem \ref{t.ContDiazPomm06}
and
   comments after Corollary \ref{c.CowePomm82}.

\begin{corollary}
\label{c.bDWp}
Let
   $f \in H (\D)$ be the generator of a semigroup with the Denjoy-Wolff point
   $a = 1$,
and let
   $\zeta_1, \ldots, \zeta_N$ be boundary regular null points of $f$
   different from $a$,
i.e.
   $f (\zeta_n) = 0$ and
   $f'(\zeta_n) < 0$.
Then
$$
   \sum_{n=1}^N \frac{1 - \Re \zeta_n}{|f' (\zeta_n)|}
 \leq - \Re \Big( \frac{1}{f (0)} \Big)
$$
and equality holds if and only if
$$
   \frac{1}{f (z)}
 = \sum_{n=1}^{N}
   \frac{1}{|f' (\zeta_n)|}
   \Big( \frac{1}{z-1} - \frac{1}{z-\zeta_n} \Big)
 + \frac{\imath c}{(z-1)^2}
$$
for some $c \in \mathbb{R}$.
\end{corollary}

If $c = 0$ then $f$ is of hyperbolic type.
Otherwise, if $c \neq 0$, then $f$ is of parabolic type.

\begin{proof}
By Theorem \ref{t.bDWp}, there exists a holomorphic function $p$ in $\D$ with
nonnegative real part, such that
$$
   \frac{1}{f (z)}
 = \sum_{n=1}^{N}
   \frac{1}{|f' (\zeta_n)|}
   \Big( \frac{1}{z-1} - \frac{1}{z-\zeta_n} \Big)
 - \frac{p (z)}{(z-1)^2},
$$
whence
$$
 - \frac{1}{f (0)}
 = \sum_{n=1}^{N}
   \frac{1 - \bar{\zeta}_n}{|f' (\zeta_n)|}
 + p (0).
$$
If $\Re p (0) > 0$, then strict inequality in Corollary \ref{c.bDWp} holds.
Otherwise, we get
   $p (z) = - \imath c$
with a real constant $c$, which is due to the Maximum Principle.
\end{proof}

Regarding hyperbolic type generators, we are now in a position to complete the
part 2) of Corollary \ref{c.CowePomm82}.

\begin{theorem}
\label{t.completionCowePomm82}
Let $f$ be the generator of a hyperbolic type semigroup with the Denjoy-Wolff
point $a = 1$, i.e.
   $f (1) = 0$ and
   $m := f' (1) > 0$.
Suppose
   $\zeta_1, \ldots, \zeta_N$ are boundary regular null points of $f$ different
   from $a$.
Then
$$
   \frac{1}{m}
 \geq
   \sum_{n=1}^N \frac{1}{|m_n|},
$$
where $m_n = f' (\zeta_n)$.
Moreover, equality holds if and only if
\begin{equation}
\label{eq.ef}
   \frac{1}{f (z)}
 = \sum_{n=1}^{N}
   \frac{1}{|m_n|}
   \Big( \frac{1}{z-1} - \frac{1}{z-\zeta_n} \Big).
\end{equation}
\end{theorem}

Another way of stating the extremality condition is to state that $1/f (z)$ is
a linear combination of
   $1/g_1 (z), \ldots, 1/g_N (z)$,
where
   $g_n (z)$ is the generator of the group of hyperbolic automorphisms having
   boundary fixed points $a = 1$ and $\zeta_n$, with $g_n' (a) = 1$ and
                                                     $g_n' (\zeta_n) = -1$.
The coefficients of the linear combination are $1/|m_n|$ and their sum is
                                               $1/m$.

\begin{proof}
Consider the holomorphic function $\sigma$ on $\D$ defined by the differential
equation
\begin{equation}
\label{eq.differential}
   f (z) \sigma' (z) = m\, \sigma (z)
\end{equation}
under the conditions $\sigma (1) = 0$ and
                     $\sigma (0) = 1$.

It was shown in
   \cite{ElinReicShoi01}
(see also \cite{Shoi01}) that $\sigma (z)$ is a univalent function starlike
with respect to the boundary point $0 = \sigma (1)$, and the formula
$$
   F_t (z) := \sigma^{-1} \Big( \exp (- f' (1) t)\, \sigma (z) \Big)
$$
for $t \geq 0$ reproduces the semigroup generated by $f$.
Furthermore,
   by \cite{ElinShoi06},
the smallest corner $C$ containing $\sigma (\D)$ is of angle $Q_a \pi$, where
   $Q_a$
is the limit of the Visser-Ostrowski quotient at the point $a = 1$.
Namely,
$$
   Q_a := \angle \lim_{z \to a} (z-a)\, \frac{\sigma' (z)}{\sigma(z)}.
$$
Since, by (\ref{eq.differential}),
$
   \displaystyle
   Q_a = \frac{m}{f' (1)} = 1,
$
we readily conclude that $\sigma (\D)$ is contained in a half-plane.

On the other hand, for every $n = 1, \ldots, N$, the image $\sigma (\D)$
contains a corner $C_n$ of angle $|Q_{\zeta_n}| \pi$ with
\begin{eqnarray*}
   Q_{\zeta_n}
 & := &
   \angle \lim_{z \to \zeta_n} (z-\zeta_n)\, \frac{\sigma' (z)}{\sigma (z)}
\\
 & = &
   \frac{m}{m_n},
\end{eqnarray*}
and different corners $C_n$ and
                      $C_m$
do not meet each other.
Since the union of $C_1, \ldots, C_N$ belongs to $C$, it follows that
\begin{equation}
\label{eq.geometry}
   \sum_{n=1}^{N} \frac{m}{|m_n|} \pi \leq \pi,
\end{equation}
proving the inequality in the theorem.

From Lemma 3 of \cite{ElinShoiZalc08} one sees that each function $\sigma$
starlike with respect to the boundary point $0 = \sigma (1)$ and with image
in a half-plane admits the representation
$$
   \sigma (z)
 = (1-z)
   \exp \Big( - \oint_{|\zeta| = 1} \log (1 - \bar{\zeta} z) d \mu (\zeta)
        \Big)
$$
with $\mu$ being a probability measure on the circle.
Moreover, if there is a boundary point $\zeta_1 \neq 1$, such that
   $Q_{\zeta_1}$ exists and is finite,
then by the same Lemma 3 of \cite{ElinShoiZalc08}
$$
   \sigma (z)
 = (1 - z)
   (1 - \bar{\zeta}_1 z)^{Q_{\zeta_1}}
   \exp \Big( - (1 + Q_{\zeta_1})
              \oint_{|\zeta| = 1} \log (1 - \bar{\zeta} z) d \mu_1 (\zeta)
        \Big),
$$
where $\mu_1$ is a probability measure on $\partial \D$. Iterating
this procedure $N$ times we arrive at the representation
$$
   \sigma (z)
 = (1-z)
   \prod_{n=1}^{N} (1 - \bar{\zeta}_n z)^{Q_{\zeta_n}}\,
   \exp \Big( \displaystyle
              - (1 + \sum_{n=1}^{N} Q_{\zeta_n})
              \oint_{|\zeta| = 1} \log (1 - \bar{\zeta} z) d \mu_N (\zeta)
        \Big)
$$
valid in our situation.

The equality in (\ref{eq.geometry}) means that
$
   \displaystyle
   1 + \sum_{n=1}^{N} Q_{\zeta_n} = 0,
$
and so
$$
   \sigma (z)
 = (1-z) \prod_{n=1}^{N} (1 - \bar{\zeta}_n z)^{Q_{\zeta_n}}
$$
(note in passing that in this case $\sigma(\D)$ is a half-plane with $N$ slits
 corresponding to the points $\zeta_1, \ldots, \zeta_N$).
Consequently,
   by (\ref{eq.differential}),
$f$ must take form (\ref{eq.ef}).
\end{proof}

In the hyperbolic case, Corollary~\ref{c.bDWp} and
                        Theorem~\ref{t.completionCowePomm82}
meet each other which can be formulated in a unified form.
Namely, let
   $f \in H (\D)$
be a generator of hyperbolic type semigroup, with $f (1) = 0$
                                              and $f' (1) < 0$,
and let
   $\zeta_1, \ldots, \zeta_N$
be points on the circle, such that $f (\zeta_n) = 0$ and
                                   $f' (\zeta_n) < 0$
for all $n = 1, \ldots, N$.
Then
\begin{eqnarray*}
   \sum_{n=1}^N \frac{1}{|f' (\zeta_n)|}
 & \leq
 & \frac{1}{f' (1)},
\\
   \sum_{n=1}^N \frac{1 - \Re \zeta_n}{|f' (\zeta_n)|}
 & \leq
 & - \Re \frac{1}{f (0)}
\end{eqnarray*}
and equality in either inequality is achieved if and only if
$$
   \frac{1}{f (z)}
 = \sum_{n=1}^N
   \frac{1}{|f' (\zeta_n)|}
   \Big( \frac{1}{z - 1} - \frac{1}{z - \zeta_n} \Big).
$$

We are now able to show that the extremal functions in
   (\ref{eq.CowePomm82}) and
   Theorem \ref{t.sharp}
are actually the same.

\begin{theorem}
\label{t.same}
Suppose
   $F$ is a holomorphic self-mapping of the unit disk, such that $F (1) = 1$
   and the angular derivative $F' (1)$ is finite.
Let moreover
   $\zeta_1, \ldots, \zeta_N$ be points on $\partial \D$, with
   $F (\zeta_n) = w\in\partial\mathbb{D}$ for some $w \neq 1$ and $F' (\zeta_n)$ being finite for
   all $n = 1, \ldots, N$.
Then
$$ 
   \sum_{n=1}^N \frac{1}{|F' (\zeta_n)|}
 \leq
   \Re \frac{w + F (0)}{w - F (0)}
 \ \ \ \mbox{and}\ \ \
   \sum_{n=1}^N \frac{1}{|F' (\zeta_n)|} \frac{1 - \Re w}{1 - \Re \zeta_n}
 \leq
   F' (1)
$$ 
and equality in either inequality holds if and only if
$
   \displaystyle
   F (z) = \frac{1 + (1 - w) A (z)}{1 - (1 - \bar{w}) A (z)},
$
where
$$
   A (z)
 = \sum_{n=1}^N
   \frac{1}{|F' (\zeta_n)|}
   \frac{(1 - z)^2}{|1 - \zeta_n|^2}
   \Big( \frac{1}{z-1} - \frac{1}{z - \zeta_n} \Big).
$$
\end{theorem}

\begin{proof}
Consider the function $f$ defined by
$$
   f (z)
 =
 - \frac{(1 - z)^2}{1 - F (z)}\,
   \frac{w - F (z)}{w - 1}
$$
for $z \in \D$.
It follows by the Berkson-Porta formula that $f$ is a holomorphic generator
on $\D$.
Furthermore,
\begin{eqnarray*}
   \angle \lim_{z \to 1} \frac{f (z)}{z - 1}
 & = &
   \angle \lim_{z \to 1} \frac{1 - z}{1 - F (z)}\, \frac{w - F (z)}{w - 1}
\\
 & = &
   \frac{1}{F' (1)}
\\
 & =: &
   m,
\end{eqnarray*}
$m$ being a positive number.
Therefore, $f$ generates a semigroup of hyperbolic type with the Denjoy-Wolff
point $a = 1$.
We also get
$$
   f (0)
 =
 - \frac{1}{1 - F (0)}\,
   \frac{w - F (0)}{w - 1}
$$
and
   $f (\zeta_n) = 0$
for all $n = 1, \ldots, N$, with finite angular derivatives
\begin{eqnarray*}
   f' (\zeta_n)
 & = &
   \angle \lim_{z \to \zeta_n} \frac{f (z)}{z - \zeta_n}
\\
 & = &
   - \frac{(1 - \zeta_n)^2}{(1 - w)^2}\, F' (\zeta_n)
\\
 & =: &
   m_n.
\end{eqnarray*}

Applying the comments after the proof of Theorem \ref{t.completionCowePomm82}
we obtain
\begin{equation}
\label{eq.atc}
   \sum_{n=1}^N
   \frac{(1 - \Re \zeta_n) |1 - w|^2}{|1 - \zeta_n|^2 |F' (\zeta_n)|}
 \leq
   \Re \Big( (1 - F (0))\, \frac{w - 1}{w - F (0)} \Big),
\end{equation}
where equality holds if and only if
\begin{equation}
\label{eq.ehiaoi}
 - \frac{1 - F (z)}{(1 - z)^2}\,
   \frac{w - 1}{w - F (z)}
 =
   \sum_{n=1}^N
   \frac{1}{|m_n|}
   \left( \frac{1}{z-1} - \frac{1}{z-\zeta_n} \right).
\end{equation}
Simplifying (\ref{eq.atc}) yields
\begin{eqnarray*}
   \sum_{n=1}^N
   \frac{1}{2}
   \frac{|1 - w|^2}{|F' (\zeta_n)|}
 & \leq &
   \Re
   \Big(
   (w - 1)\, \frac{(1 - F (0))(\bar{w} - \overline{F (0)})}{|w - F (0)|^2}
   \Big)
\\
 & = &
   \Re \Big( (1 - w)\, \frac{1 - |F (0)|^2}{|w - F (0)|^2} \Big)
\\
 & = &
   (1 - \Re w)\, \Re \frac{w + F (0)}{w - F (0)},
\end{eqnarray*}
which is precisely the first inequality stated in the theorem.

On the other hand, by the same comments,
$$
   \frac{1}{m} \geq \sum_{n=1}^N \frac{1}{|m_n|}
$$
where equality is valid if and only if $F$ satisfies (\ref{eq.ehiaoi}).
This is equivalent to the second inequality of the theorem.

To complete the proof it remains to express the extremal function $F$ from
(\ref{eq.ehiaoi}).
To this end, we denote $R$ for the right-hand side of (\ref{eq.ehiaoi}).
A trivial verification shows that
$$
   F (z)
 = \frac{\displaystyle 1 - \frac{R}{1-w} w (1-z)^2}
        {\displaystyle 1 - \frac{R}{1-w} (1-z)^2}.
$$
Substituting the formulas for $m_n$, we get
$$
   \frac{R}{1-w}
 = (1-\bar{w})
   \sum_{n=1}^N
   \frac{1}{|F' (\zeta_n)|}
   \frac{1}{|1 - \zeta_n|^2}
   \Big( \frac{1}{z-1} - \frac{1}{z - \zeta_n} \Big),
$$
which is the desired formula.
\end{proof}

For the dilation case one cannot assume any representation of the form
   (\ref{eq.preliminary})
where $g_n$ are generators of groups, for there is no group generator having
both interior and boundary null points.
On the other hand, the generators of hyperbolic groups are particular cases
of generators of the form
   $(z-a) (1 - \bar{a} z) R (z)$
where $R (z)$ is an affine-fractional transformation of $\D$ onto
the right half-plane.
This leads us to a generalization of (\ref{eq.bDWp}) even when the
Denjoy-Wolff point $a$ is located inside the disk $\D$.

\section{Separation of singularities in the general case}
\label{s.separation}

Let
   $f (z) = (z-a) (1 - \bar{a} z) p (z)$
with $a \in \overline{\D}$.
Suppose that $f (\zeta) = 0$ and
             $f' (\zeta) < 0$
for some $\zeta \in \partial \D$ different from $a$.
Then
\begin{eqnarray}
\label{eq.f'(zeta)}
   f' (\zeta)
 & = &
   \angle \lim_{z \to \zeta}
   \frac{f (z)}{z-\zeta}
\nonumber
\\
 & = &
 - \angle \lim_{z \to \zeta}
   \frac{(z-a) (1 - \bar{a} z) \bar{\zeta}}
        {\displaystyle (1 - \bar{\zeta} z) \frac{1}{p (z)}}
\nonumber
\\
 & = &
 - \frac{|1 - \bar{\zeta} a|^2}{2 \ell},
\end{eqnarray}
where
$
   \displaystyle
   \ell :=
 \angle \lim_{z \to \zeta} \frac{1}{2}\, (1 - \bar{\zeta} z) \frac{1}{p (z)}.
$
By Lemma \ref{l.JWC},
$$
   \Re \frac{1}{p (z)} \geq \ell\, \frac{1-|z|^2}{|z-\zeta|^2}.
$$

Now pick any $\eta \in \partial \D$ different from $\zeta$, and consider the
function
$$
   R (z)
 = \frac{\eta - z}{(1 - \bar{\zeta} \eta) (z - \zeta)}
$$
for $z \in \D$.
An easy computation shows that $R$ maps $\D$ onto the right half-plane
   $\{ \Re w > 0 \}$
and
$$
   \Re R (z) = \frac{1}{2}\, \frac{1-|z|^2}{|z-\zeta|^2}.
$$
Hence,
$
   \displaystyle
   \Re \frac{1}{p (z)} \geq 2 \ell\, \Re R (z).
$

Let $g \in H (\D)$ be the generator in $\D$ defined by
$
   \displaystyle
   g (z) = \frac{(z-a) (1 - \bar{a} z)}{2 \ell\, R (z)}.
$
Then we get
\begin{eqnarray*}
   \frac{1}{g (z)}
 & = &
   2 \ell\,
   \frac{\eta - z}
        {(1 - \bar{\zeta} \eta) (z - \zeta) (z-a) (1 - \bar{a} z)}
\\
 & = &
   \frac{2 \ell}{1 - \bar{\zeta} \eta}
   \Big( \frac{A}{z - \zeta} -  \frac{B}{z - a} \Big)
\end{eqnarray*}
where
$$
   A (z - a) - B (z - \zeta) = \frac{\eta - z}{1 - \bar{a} z},
$$
or
\begin{equation}
\label{eq.AB}
   (A - B) z - A a + B \zeta = \frac{\eta - z}{1 - \bar{a} z}.
\end{equation}

We distinguish two cases:

1)
If $a = 1$, then $\eta = 1$ and so
   $(A - B) z - A a + B \zeta = 1$.
Hence, $A = B$ with
$
   \displaystyle
   A = \frac{1}{\zeta - 1}
$
and
\begin{eqnarray*}
   \frac{1}{g (z)}
 & = &
   \frac{2 \ell}{(1 - \bar{\zeta}) (\zeta - 1)}
   \Big( \frac{1}{z - \zeta} -  \frac{1}{z - 1} \Big)
\\
 & = &
   \frac{2 \ell}{|1 - \zeta|^2}
   \Big( \frac{1}{z - 1} -  \frac{1}{z - \zeta} \Big)
\\
 & = &
   \frac{1}{|f' (\zeta)|}
   \Big( \frac{1}{z - 1} -  \frac{1}{z - \zeta} \Big),
\end{eqnarray*}
which is due to (\ref{eq.f'(zeta)}).

2)
If $a = 0$, then (\ref{eq.AB}) reduces to
   $(A - B) z + B \zeta = \eta - z$.
Hence,
$$
   \begin{array}{rcl}
     A - B
   & =
   & - 1,
\\
     B \zeta
   & =
   & \eta,
   \end{array}
$$
or
   $B = \bar{\zeta} \eta$ and
   $A = \bar{\zeta} \eta - 1$.
It follows that
\begin{eqnarray*}
   \frac{1}{g (z)}
 & = &
   \frac{2 \ell}{1 - \bar{\zeta} \eta}
   \Big( \frac{\bar{\zeta} \eta - 1}{z - \zeta} - \frac{\bar{\zeta} \eta}{z}
   \Big)
\\
 & = &
   \frac{1}{|f' (\zeta)|}
   \Big( \frac{\bar{\zeta} \eta}{\bar{\zeta} \eta - 1} \frac{1}{z}
       - \frac{1}{z - \zeta}
   \Big).
\end{eqnarray*}

Note that in both the cases $a \in \partial \D$ and
                            $a \in \D$
the function $h$ defined by the equality
$$
   \frac{1}{f (z)} - \frac{1}{g (z)} = \frac{1}{h (z)}
$$
is also a holomorphic generator in $\D$ of the same type as $f$,
   provided $f \neq g$.
Indeed, rewriting this equalitty in the form
$$
   \frac{1}{(z-a) (1 - \bar{a} z) p (z)}
 - \frac{2 \ell\, R (z)}{(z-a) (1 - \bar{a} z)}
 = \frac{1}{(z-a) (1 - \bar{a} z) q (z)}
$$
with
   $h (z) = (z-a) (1 - \bar{a} z) q (z)$,
we deduce from
   $\Re (1/p (z)) \geq 2 \ell\, \Re R (z)$
that
   $\Re (1/q (z)) \geq 0$,
and so
   $\Re q (z) \geq 0$ for all $z \in \D$.
The desired assertion follows from Theorem \ref{t.generator}.

\begin{remark}
\label{r.q(z)}
Except when $f$ is hyperbolic the function $q (z)$ may be constant, i.e.
   $\Re q (z) = 0$.
In this situation $h$ is a generator of a group of either
parabolic or elliptic automorphisms.
\end{remark}

Arguing by induction in much the same way as in the proof of
   Theorem \ref{t.bDWp},
we obtain the following general assertion.

\begin{theorem}
\label{t.general}
Let
   $f \in H (\D)$ be the generator of a semigroup with a Denjoy-Wolff point
   $a \in \overline{\D}$.
If moreover $\zeta_1, \ldots, \zeta_N$ are boundary regular null points of $f$
different from $a$ and satisfying $f' (\zeta_n) < 0$ for all $n = 1,\ldots,N$,
   then there exists a number $r \geq 0$, such that
$$
   \frac{1}{f (z)}
 = \sum_{n=1}^{N}
   \frac{1}{g_n (z)}
 + \frac{r}{h (z)},
$$
where
   $g_n$ are generators of the form $g_n (z) = (z-a) (1 - \bar{a} z) R_n (z)$
   with $R_n$ being affine-fractional transformations of $\D$ onto the (open)
   right half-plane, $R_n (\zeta_n) = 0$,
and
   $h$ is a holomorphic generator of the same type as $f$.
\end{theorem}

If $a = 0$, we describe the generators $g_n (z)$ more explicitly by obtaining
a singularities separation theorem for dilation type generators.

\begin{corollary}
\label{c.dilation}
Suppose
   $f \in H (\D)$ is the generator of a semigroup with the Denjoy-Wolff point
   $a = 0$
and
   $\zeta_1, \ldots, \zeta_N$ are boundary regular null points of $f$.
Given any $\eta_1, \ldots, \eta_N \in \partial \D$ satisfying
   $\eta_n \neq \zeta_n$ for all $n = 1, \ldots, N$,
there is an $r \geq 0$, such that
$$
   \frac{1}{f (z)}
 = \sum_{n=1}^{N}
   \frac{1}{|f' (\zeta_n)|}
   \Big( \frac{\bar{\zeta}_n \eta_n}{\bar{\zeta}_n \eta_n - 1} \frac{1}{z}
       - \frac{1}{z - \zeta_n}
   \Big)
 + \frac{r}{h (z)},
$$
where
   $h$ is a holomorphic generator with $h (0) = 0$ and
$
   \displaystyle
   \angle \lim_{z \to \zeta_n} \frac{h (z)}{z - \zeta_n} = \infty.
$
\end{corollary}

Note that $\Re f' (0) > 0$.
Indeed, otherwise $f$ is of the form
   $f (z) = \imath c z$
for some $c \in \mathbb{R}$, hence,
   it cannot have boundary regular null points.

In particular, choosing $\eta_n = - \zeta_n$ yields a simpler formula
\begin{equation}
\label{eq.simpler}
   \frac{1}{f (z)}
 = \sum_{n=1}^{N}
   \frac{1}{|f' (\zeta_n)|}
   \Big( \frac{1}{2z}  - \frac{1}{z - \zeta_n}
   \Big)
 + \frac{r}{h (z)}.
\end{equation}
This representation implies that
$
   \displaystyle
   \Im \frac{1}{f' (0)} = \Im \frac{r}{h' (0)}.
$

\begin{corollary}
\label{c.simpler}
Let
   $f \in H (\D)$ be the generator of a semigroup with the Denjoy-Wolff point
   $a = 0$
and let
   $\zeta_1, \ldots, \zeta_N$ be boundary regular null points of $f$.
Then
$$
   \sum_{n=1}^{N}
   \frac{1}{|f' (\zeta_n)|}
 \leq
   2\, \Re \frac{1}{f' (0)}
$$
and equality holds if and only if
$$
   \frac{1}{f (z)}
 = \sum_{n=1}^{N}
   \frac{1}{|f' (\zeta_n)|}
   \Big( \frac{1}{2z}  - \frac{1}{z - \zeta_n}
   \Big)
 + \frac{\imath c}{z},
$$
where
   $c \in \mathbb{R}$.
\end{corollary}

Returning to the boundary interpolation problem mentioned in
   Sections~\ref{s.Introduction} and
            \ref{s.roCPti},
we are now in a position to solve it by using Theorems \ref{t.bDWp}
                                                   and \ref{t.general}.
In particular, let
   $\zeta_1, \ldots, \zeta_{N+1}$
be pairwise different points on the circle $\partial \mathbb{D}$
and
   $m_1, \ldots, m_{N+1},$
be given real numbers satisfying $m_n < 0$ for $n = 1, \ldots, N$ and
                                 $m_{N+1} > 0$.
In order that there might exist a holomorphic generator $f$ satisfying
   $f (\zeta_n) = 0$ and
   $f' (\zeta_n) = m_n$
for all $n = 1, \ldots, N+1$, it is necessary and sufficient that the numbers
$m_n$ would satisfy
$$
   \sum_{n=1}^N \frac{1}{|m_n|} \leq \frac{1}{m_{N+1}}.
$$
If this condition holds then the solution can be represented by the formula
$$
   f (z)
 = \Big( \sum_{n=1}^N
         \frac{1}{|m_n|}
         \Big( \frac{1}{z-\zeta_{N+1}} - \frac{1}{z-\zeta_n} \Big)
       + \frac{r p (z)}{(z - \zeta_{N+1}) (1 - z \bar{\zeta}_{N+1})}
   \Big)^{-1},
$$
where either $r = 0$ or
             $r > 0$,
and $p$ is any function in $H (\D)$ with
   $\Re p (z) \geq 0$ and
$$
   \angle \lim_{z \to \zeta_n} (z-\zeta_n) p (z) = 0
$$
for $n = 1, \ldots, N$.

\section{Distortion theorems}
\label{s.distortion}

In this section we establish distortion theorems for generators of all types.
They can be thought of as quantitative improvements of
   results in Section~\ref{s.roCPti}
as well as
   Corollary~\ref{c.bDWp} and
   Theorem~\ref{t.completionCowePomm82}.
In fact, these theorems cover all inequalities given in the previous sections
                                                 and in \cite{ContDiazPomm06},
and present extremal functions fulfilling the corresponding equalities.
It should be mentioned that the key tools for the distortion theorems are
Theorems \ref{t.bDWp},
         \ref{t.general}
and Corollary \ref{c.dilation}, they are thus of basic importance.

\begin{theorem}
\label{t.distortion}
Let
   $f \in H (\D)$ be the generator of a semigroup with the Denjoy-Wolff point
   $a = 1$.
Suppose
   $\zeta_1, \ldots, \zeta_N$ are boundary regular null points of $f$
   different from $a$, with $f' (\zeta_n) < 0$.
Then
$$
   C\, l (z)
 \leq
   \Big| \frac{1}{f (z)} - \sum_{n=1}^{N} \frac{1}{g_n (z)}
                         - \frac{\imath c}{(z-1)^2}
   \Big|
 \leq
   C\, u (z)
$$
where
$
   \displaystyle
   \frac{1}{g_n (z)}
 = \frac{1}{|f' (\zeta_n)|}
   \Big( \frac{1}{z-1} - \frac{1}{z-\zeta_n} \Big)
$
and
$$
   \begin{array}{rclcrcl}
     c
   & =
   & \displaystyle
       \Im \Big( \frac{1}{f (0)} \Big)
     + \sum_{n=1}^{N} \frac{\Im \zeta_n}{|f' (\zeta_n)|},
   &
   & C
   & =
   & \displaystyle
     - \Re \Big( \frac{1}{f (0)} \Big)
     - \sum_{n=1}^{N} \frac{1 - \Re \zeta_n}{|f' (\zeta_n)|},
\\
     u (z)
   & =
   & \displaystyle
     \frac{1 + |z|}{1 - |z|}
     \frac{1}{|1 - z|^2},
   &
   & l (z)
   & =
   & \displaystyle
     \frac{1 - |z|}{1 + |z|}
     \frac{1}{|1 - z|^2}.
   \end{array}
$$
\end{theorem}

\begin{proof}
If
$$
   \frac{1}{f (z)}
 = \sum_{n=1}^{N} \frac{1}{g_n (z)} + \frac{\imath c}{(z-1)^2},
$$
then
   the right-hand inequality is obvious and
   the left-hand inequality follows from Corollary~\ref{c.bDWp}.

Otherwise we deduce from Theorem \ref{t.bDWp} that
\begin{equation}
\label{eq.1/f(z)}
   \frac{1}{f (z)} - \sum_{n=1}^{N} \frac{1}{g_n (z)}
 = \frac{r}{h (z)},
\end{equation}
where
$$
   r = 1 - f' (1) \sum_{n=1}^{N} \frac{1}{|f' (\zeta_n)|}
     > 0
$$
and $h$ is a generator of the same type as $f$, with $h' (1) = f' (1)$.
So we have to estimate
$$
   \Big| \frac{r}{h (z)} - \frac{\imath c}{(z-1)^2} \Big|
 = \Big| \frac{r p (z) - \imath c}{(z-1)^2} \Big|
$$
for some holomorphic function $p$ with nonnegative real part in $\D$,
   cf. Theorem \ref{t.generator}.
By (\ref{eq.1/f(z)}),
\begin{eqnarray*}
   \Im \left( r p (0) - \imath c \right)
 & = &
   \Im \Big( \frac{1}{f (0)}
           - \sum_{n=1}^{N} \frac{\bar{\zeta}_n - 1}{|f' (\zeta_n)|}
       \Big)
 - c
\\
 & = &
   0.
\end{eqnarray*}
Therefore, the Riesz-Herglotz formula leads to the estimate
$$
   \frac{1 - |z|}{1 + |z|}\, \Re \left( r p (0) \right)
 \leq
   |r p (z) - \imath c|
 \leq
   \frac{1 + |z|}{1 - |z|}\, \Re \left( r p (0) \right)
$$
for all $z \in \D$.
Since
$$
   \Re \left( r p (0) \right)
 = - \Re \Big( \frac{1}{f (0)} \Big)
   - \sum_{n=1}^{N} \frac{1 - \Re \zeta_n}{|f' (\zeta_n)|},
$$
the proof is complete.
\end{proof}

Using our Theorem~\ref{t.bDWp} and
          formula (2.1) of \cite{Shoi03},
one obtains the distortion estimate
$$
   \Big| \frac{1}{f (z)} - \sum_{n=1}^{N}
                           \frac{1}{|f' (\zeta_n)|}
                           \Big( \frac{1}{z-1} - \frac{1}{z-\zeta_n} \Big)
   \Big|
 \leq
   \frac{2 \mathit{\Delta}}{1 - |z|^2},
$$
where
$
   \displaystyle
   \mathit{\Delta}
 = \frac{1}{f' (1)} - \sum_{n=1}^{N} \frac{1}{|f' (\zeta_n)|}.
$
We now prove a sharper estimate.

\begin{theorem}
\label{t.dh}
Let
   $f$ be the generator of a hyperbolic semigroup with the Denjoy-Wolff point
   $a = 1$,
i.e. $f (1) = 0$ and
     $m := f' (1) > 0$.
Suppose
   there are boundary regular points $\zeta_1, \ldots, \zeta_n$ of $f$
   different from $a$, with $m_n := f' (\zeta_n) < 0$.
Then
$$
   \Big| \frac{1}{f (z)} - \sum_{n=1}^{N}
                           \frac{1}{|f' (\zeta_n)|}
                           \Big( \frac{1}{z-1} - \frac{1}{z-\zeta_n} \Big)
                         - \frac{\mathit{\Delta}}{(1 - |z|^2)}
                           \frac{1-\bar{z}}{1-z}
   \Big|
 \leq
   \frac{\mathit{\Delta}}{1 - |z|^2}
$$
for all $z \in \D$, where
$
   \displaystyle
   \mathit{\Delta}
 = \frac{1}{m} - \sum_{n=1}^{N} \frac{1}{|m_n|}.
$
\end{theorem}

\begin{proof}
By Theorem \ref{t.bDWp},
$$
   \frac{1}{f (z)}
 = \sum_{n=1}^{N}
   \frac{1}{|f' (\zeta_n)|}
   \Big( \frac{1}{z-1} - \frac{1}{z-\zeta_n} \Big)
 - \frac{1}{(1-z)^2 q (z)}
$$
for some holomorphic function $q$ of nonnegative real part in $\D$.
Hence it follows that
\begin{eqnarray*}
   \angle \lim_{z \to a}
   \frac{1}{(1-z) q (z)}
 & = &
   \frac{1}{m} - \sum_{n=1}^{N} \frac{1}{|m_n|}
\\
 & = &
   \mathit{\Delta}.
\end{eqnarray*}
Applying Lemma \ref{l.JWC} with $\zeta = 1$ and
                                $\ell = 1/2 \mathit{\Delta}$,
we get
$$
   \Re q (z)
 \geq
   \frac{1}{2 \mathit{\Delta}}\, \frac{1 - |z|^2}{|z-1|^2}
$$
whence
$$
   \Big| - \frac{1}{q (z)} + \mathit{\Delta}\, \frac{|z-1|^2}{1 - |z|^2}
   \Big|
 \leq
   \mathit{\Delta}\, \frac{|z-1|^2}{1 - |z|^2}
$$
for all $z \in \D$.
Substituting
$$
   - \frac{1}{q (z)}
 =
   \frac{(1-z)^2}{f (z)} - \sum_{n=1}^{N}
                           \frac{(1-z)^2}{|f' (\zeta_n)|}
                           \Big( \frac{1}{z-1} - \frac{1}{z-\zeta_n} \Big)
$$
into this inequality yields the desired estimate.
\end{proof}

In the dilation case, we can proceed in the same way to derive a quantitative
version of Corollary \ref{c.dilation}.

\begin{theorem}
\label{t.dd}
Assume that
   $f \in H (\D)$ is the generator of a semigroup with the Denjoy-Wolff point
   $a = 0$.
If
   $\zeta_1, \ldots, \zeta_N$ are boundary regular null points of $f$,
then one has
$$
   C\, \frac{1 - |z|}{1 + |z|}
 \leq
   \Big|
   z \Big( \frac{1}{f (z)}
         - \sum_{n=1}^{N}
           \frac{1}{|f' (\zeta_n)|}
           \Big( \frac{1}{2z}  - \frac{1}{z - \zeta_n} \Big)
     \Big)
 - \imath \Im\, \frac{1}{f' (0)}
   \Big|
 \leq
   C\,
   \frac{1 + |z|}{1 - |z|},
$$
where
$
   \displaystyle
   C
 = \Re \frac{1}{f' (0)} - \frac{1}{2} \sum_{n=1}^N \frac{1}{|f' (\zeta_n)|}.
$
\end{theorem}

\begin{proof}
By formula (\ref{eq.simpler}), we may write
$$
   \frac{1}{f (z)}
 =
   \sum_{n=1}^{N}
   \frac{1}{|f' (\zeta_n)|}
   \Big( \frac{1}{2z}  - \frac{1}{z - \zeta_n} \Big)
 + \frac{r}{h (z)},
$$
$h (z)$ being a holomorphic generator on $\D$.
Denote
$
   \displaystyle
   c
 := \Im \frac{1}{f' (0)}
  = \Im \frac{r}{h' (0)}
$
and consider
$$
   k (z)
 := z\, \frac{r}{h (z)} - \imath c.
$$

Obviously,
   $\Re k (z) \geq 0$
and
   $\Im k (0) = 0$.
If $\Re k (0) = 0$, then our assertion follows by the Maximum Principle.
Otherwise there is a holomorphic generator $g (z)$, such that
$$
   \frac{r}{h (z)} = \frac{\imath c}{z} + \frac{1}{g (z)}
$$
and the number $g' (0)$ is real.
Thus,
$$
   z \Big( \frac{1}{f (z)}
         - \sum_{n=1}^{N}
           \frac{1}{|f' (\zeta_n)|}
           \Big( \frac{1}{2z} - \frac{1}{z - \zeta_n} \Big)
     \Big)
 - \imath c
 = \frac{z}{g (z)}.
$$

Set
$
   \displaystyle
   P (z)
 = \frac{z}{f (z)}
$
and
$
   \displaystyle
   Q (z)
 = \frac{z}{g (z)},
$
then
$$
   P (0)
 - \frac{1}{2}
   \sum_{n=1}^{N}
   \frac{1}{|f' (\zeta_n)|}
 - \imath c
 = Q (0)
$$
whence
\begin{eqnarray*}
   \Re Q (0)
 & = &
   \Re P (0)
 - \frac{1}{2}
   \sum_{n=1}^{N}
   \frac{1}{|f' (\zeta_n)|},
\\
   \Im Q (0)
 & = &
   \Im P (0)
 - \Im \frac{1}{f' (0)},
\end{eqnarray*}
that is, $\Im Q (0) = 0$.
Now it follows by the Riesz-Herglotz formula that
$$
   \Re Q (0)\, \frac{1 - |z|}{1 + |z|}
 \leq
   |Q (z)|
 \leq
   \Re Q (0)\, \frac{1 - |z|}{1 + |z|},
$$
which establishes the theorem.
\end{proof}

\section{The Cowen-Pommerenke inequalities revisited}
\label{s.tCPir}

In this section we look more closely at Theorem
\ref{t.CowePomm82}. Our objective is to show that the inequality
in assertion 3) thereof still remains true for holomorphic
self-mappings $F$ of $\D$ of hyperbolic type.

\begin{theorem}
\label{t.P2} Let
   $F$ be a holomorphic self-mapping of $\D$ with the Denjoy-Wolff point
   $a = 1$,
i.e. $F (1) = 1$ and
     $F' (1) \leq 1$.
If
   $\zeta_1, \ldots, \zeta_N$ are boundary regular fixed points of $F$
   different from $1$,
then
\begin{equation}
\label{eq.gthm}
   \sum_{n=1}^N \frac{|1 - \zeta_n|^2}{F' (\zeta_n) - 1}
 \leq 2\, \Re \Big( \frac{1}{F (0)} - 1 \Big).
\end{equation}
Moreover, if $F' (1) < 1$, then equality in (\ref{eq.gthm}) holds
if and only if equality in the assertion 2) of Theorem
\ref{t.CowePomm82} holds, and this is precisely the case for $F$
of the form
\begin{equation}
\label{eq.extremalF}
   F (z)
 = C^{-1}
   \Big( C (z)
       + \frac{2}
              {\displaystyle
               (1-z) \sum_{n=1}^N
                     \frac{1}{F' (\zeta_n) - 1}
                     \frac{1 - \zeta_n}{z - \zeta_n}
              }
   \Big),
\end{equation}
where $
   \displaystyle
   C (z) = \frac{1+z}{1-z}
$ is the Cayley transform of $\D$.
\end{theorem}

\begin{proof}
Consider the holomorphic function $p$ in $\D$ defined by
\begin{equation}
\label{eq.p(z)}
   p (z)
 = \frac{1 + F (z)}{1 - F (z)} - \frac{1 + z}{1 - z}
\end{equation}
for $z \in \D$. Since $a = 1$ is the Denjoy-Wolff point of $F$ we
deduce from Julia's lemma that
   $\Re p (z) \geq 0$
for $z \in \D$. Therefore, the function
$$
   f (z) = - (1-z)^2\, p (z)
$$
is the generator of a semigroup of holomorphic self-mappings of
$\D$ with the
   Denjoy-Wolff point $a = 1$ and
   boundary regular null points $\zeta_1, \ldots, \zeta_N$.

Indeed, it is clear that $
   f (\zeta_n) = - (1-\zeta_n)^2\, p (\zeta_n)
$ vanishes and
\begin{eqnarray*}
   f' (\zeta_n)
 & = &
   \lim_{z \to \zeta_n} \frac{f (z)}{z - \zeta_n}
\\
 & = &
   -\,
   (1-\zeta_n)^2\,
   \lim_{z \to \zeta_n} \frac{p (z)}{z - \zeta_n}
\\
 & = &
   -\,
   (1-\zeta_n)^2\,
   p' (\zeta_n)
\end{eqnarray*}
for all $n = 1, \ldots, N$. Using (\ref{eq.p(z)}) yields
$$
   p' (\zeta_n)
 = (F' (\zeta_n) - 1)\, \frac{2}{(1 - \zeta_n)^2}
$$
whence
\begin{eqnarray*}
   f' (\zeta_n)
 & = &
   - 2\, (F' (\zeta_n) - 1)
\\
 & < &
   0.
\end{eqnarray*}

Furthermore,
\begin{eqnarray}
\label{eq.f'(1)}
   f' (1)
 & = &
   \lim_{z \to 1}
   (1-z)\, p (z)
\nonumber
\\
 & = &
   \lim_{z \to 1}
   \frac{1 - z}{1 - F (z)}\, (1 + F (z))
 - 2
\nonumber
\\
 & = &
   2
   \Big( \frac{1}{F' (1)} - 1 \Big)
\nonumber
\\
 & \geq &
   0.
\end{eqnarray}
Hence it follows from Theorem \ref{t.ContDiazPomm06} that
\begin{equation}
\label{eq.ContDiazPomm06m}
   \sum_{n=1}^N \frac{1}{2}\, \frac{|1 - \zeta_n|^2}{|f' (\zeta_n)|}
 \leq - \Re \Big( \frac{1}{f (0)} \Big).
\end{equation}
Since
$$
   f (0)
 = - p (0)
 = \frac{F (0) + 1}{F (0) - 1} - 1
 = 2\, \frac{F (0)}{F (0) - 1},
$$
the inequalities (\ref{eq.gthm}) and
               (\ref{eq.ContDiazPomm06m})
are equivalent. Note that (\ref{eq.f'(1)}) shows that $f$ is of
hyperbolic type if and only if so is $F$.

Assume that $F' (1) < 1$. Then $f' (1) > 0$, and so we conclude
from the second assertion of
   Corollary \ref{c.CowePomm82}
that
\begin{equation}
\label{eq.saoc}
   \sum_{n=1}^N \frac{1}{|f' (\zeta_n)|}
 \leq \frac{1}{f' (1)}.
\end{equation}
This inequality, in turn, proves to be equivalent to the
inequality of the assertion 2) of Theorem \ref{t.CowePomm82}.

As already mentioned before Theorem \ref{t.same}, equalities in
either of
   (\ref{eq.ContDiazPomm06m})
   (\ref{eq.saoc})
is achieved if and only if
$$
   \frac{1}{f (z)}
 = \sum_{n=1}^N
   \frac{1}{|f' (\zeta_n)|}
   \Big( \frac{1}{z-1} - \frac{1}{z-\zeta_n} \Big).
$$
On combining this equality with
   $f (z) = - (1-z)^2 p (z)$
we get (\ref{eq.extremalF}), as desired.
\end{proof}

The particular case of Theorem \ref{t.P2} corresponding to $N = 1$
seems surprisingly to be new.

\begin{corollary}
\label{c.P2} Suppose $F$ is a holomorphic self-mapping of $\D$,
such that
   $F (1) = 1$ and $0 < F' (1) < 1$.
Let $\zeta \neq 1$ be a boundary regular fixed point of $F$. Then
$F$ is an automorphism of $\D$ if and only if
$$
   \frac{1 - \Re \zeta}{F' (\zeta) - 1}
 = \Re \Big( \frac{1}{F (0)} - 1 \Big).
$$
\end{corollary}

Finally we note that Corollary~\ref{c.simpler} gains in interest
if we realize that Theorem \ref{t.P2} is actually its easy
consequence. Indeed, applying Corollary \ref{c.simpler} to the
dilation type generator
   $f (z) = z\, p (z)$
with
$$
   p (z)
 = \frac{1 + F (z)}{1 - F (z)} - \frac{1 + z}{1 - z}
$$
in $\D$ we get Theorem \ref{t.P2}. Recall that $- (1-z)^2 p (z)$
is the generator of a semigroup with the
   Denjoy-Wolff point $a = 1$
and
   boundary regular null points $\zeta_1, \ldots, \zeta_N$,
cf. the proof of Theorem \ref{t.P2}. We have
   $f (z) = k (z)\, (1-z)^2 p (z)$
where $k$ is the Koebe univalent function
$$
   k (z) = \frac{z}{(1-z)^2}
$$
in $\D$. One verifies immediately that
   Corollary \ref{c.simpler} for $f (z) := z\, p (z)$ is equivalent to
   Corollary \ref{c.bDWp} for $f (z) := - (1-z)^2 p (z)$.


\end{document}